\documentclass[11pt,leqno]{article}
\usepackage{graphicx, amssymb,amsmath, latexsym,cite,amsfonts, amsthm,stmaryrd}
\usepackage{amsthm}
\usepackage{mathtools}   
\usepackage{enumerate,booktabs} 
\usepackage{mathrsfs} 
\usepackage[pagebackref]{hyperref}
\hypersetup{citecolor=purple, linkcolor=blue, colorlinks=true}
\usepackage{authblk}  
\usepackage{tikz}
   
\newtheorem{theorem} {Theorem} [section]

\newtheorem{lemma} [theorem] {Lemma}
\newtheorem{corollary} [theorem] {Corollary} 
\theoremstyle{definition}
\newtheorem{definition}[theorem]{Definition}
\newtheorem{example}[theorem]{Example}
\newtheorem{remark}[theorem]{Remark}
\newtheorem*{notation}{Notation}

\renewcommand{\ge}{\geqslant}
\renewcommand{\le}{\leqslant}
\renewcommand{\geq}{\geqslant}
\renewcommand{\leq}{\leqslant}
\newcommand{\lhdeq}{\trianglelefteqslant}    
\newcommand{\defi}{=}

\usepackage[a4paper, margin=2.2cm]{geometry}
\let\originalleft\left
\let\originalright\right
\renewcommand{\left}{\mathopen{}\mathclose\bgroup\originalleft}
\renewcommand{\right}{\aftergroup\egroup\originalright}

\newcommand{\paren}[1]{\left( #1 \right)}

\newcommand{\derk}[2]{{\rm Fix}_{#1}\paren{#2}}

\newcommand{\pder}[1]{\delta\paren{#1}}
\newcommand{\pderk}[2]{\delta_{#1}\paren{#2}}
\newcommand{\alt}{\mathcal{A}}
\newcommand{\Der}{\mathrm{Der}}
\newcommand{\Fix}{\mathrm{Fix}}
\newcommand{\Sym}{\mathrm{Sym}}
\newcommand{\sym}{\mathcal{S}}
\newcommand{\stir}[2]{\genfrac{[}{]}{0pt}{}{#1}{#2}}

\newcommand{\intrans}{\times_I}
\newcommand{\product}{\times_P}
\newcommand{\imprim}{\wr_I}
\newcommand{\power}{\wr_P}
\newcommand{\Px}{{\rm({\bf Px})}}
\newcommand{\Ix}{{\rm({\bf Ix})}}
\newcommand{\Pw}{{\rm({\bf Pw})}}
\newcommand{\Iw}{{\rm({\bf Iw})}}

\author[1]{Vishnuram Arumugam\thanks{Supported by the Australian Government 
Research Training Program.}}
\author[2]{Heiko Dietrich\thanks{%
    Research visit of H. Dietrich to UWA was supported by the
    2022 Cheryl Praeger Visiting Fellowship.}}
\author[1]{S.P. Glasby\thanks{Supported by the Australian Research Council
Discovery Grant DP190100450. This problem was posed at the CMSC 2022 Annual Research Retreat. All authors thank the CMSC for its hospitality.}}
\affil[1]{\small Center for the Mathematics of Symmetry and Computation,
  University of Western Australia, Perth 6009, Australia\\
\href{mailto:Vishnuram.Arumugam@uwa.edu.au}{Vishnuram.Arumugam@uwa.edu.au} and \href{mailto:Stephen.Glasby@uwa.edu.au}{Stephen.Glasby@uwa.edu.au}}
\affil[2]{\small School of Mathematics, Monash University, Clayton 3800, Australia\\\href{mailto:heiko.dietrich@monash.edu}{Heiko.Dietrich@monash.edu}}

\title{\bf Derangements in wreath products\\ of permutation groups}

\begin{document}

\maketitle

\vspace*{-0.5cm}

\begin{abstract}
  \noindent Given a finite group $G$ acting on a set $X$
  let $\delta_k(G,X)$ denote
  the proportion of elements in $G$ that have exactly $k$ fixed points in $X$.
  Let $\sym_n$ denote the symmetric group acting
  on $[n]=\{1,2,\dots,n\}$. For $A\le\sym_m$ and $B\le\sym_n$, the permutational
  wreath product $A\wr B$ has two natural actions and we give formulas for both,
  $\delta_k(A\wr B,[m]{\times}[n])$ and  $\delta_k(A\wr B,[m]^{[n]})$. We prove
  that for $k=0$ the values of these proportions are dense in the intervals
  $[\delta_0(B,[n]),1]$ and $[\delta_0(A,[m]),1]$.  Among further results, we provide
  estimates for $\delta_0(G,[m]^{[n]})$ for subgroups $G\leq \sym_m\wr\sym_n$ containing $\alt_m^{[n]}$.
  \vskip1mm\noindent
  {\bf Keywords:} permutation groups, derangements, fixed-point-free permutations, wreath products
  \vskip1mm\noindent
  {\bf 2020 Mathematics Subject Classification:} 20B07, 20B35, 05A05
\end{abstract}


\section{Introduction}\label{sec_intro}
\noindent Let $\sym_n$ and $\alt_n$ denote the symmetric and alternating group acting on
$[n]=\{1,2,\dots,n\}$. A permutation of $\sym_n$ that fixes no element in
$[n]$ is a \emph{derangement}. The proportion of derangements
in a subset $C\subseteq\sym_n$ is denoted $\delta(C)$.
Given $k\in\{0,1,\dots,n\}$, let $\derk{k}{C}$ denote the set of
all permutations in $C$ with precisely $k$ fixed points. We write
$\delta_k(C)=|\derk{k}{C}|/|C|$, and note that $\delta(C)=\delta_0(C)$.
{The set of derangements in $C$ is denoted by $\Der(C)$ or $\derk{0}{C}$.}
Given subgroups $A\le\sym_m$ and
$B\le \sym_n$, the wreath product $A\wr B=A^{[n]}\rtimes B$ gives rise to two
natural permutation subgroups: the \emph{imprimitive} subgroup
$A\imprim B\le\sym_{mn}$ and the subgroup $A\power B\le\sym_{m^n}$ 
with \emph{power} (or \emph{product}) action, see \Iw\ and \Pw\ in Section \ref{sec_actions} for details. There is a large body of literature on derangements; we highlight the most relevant results in Section \ref{sec_background}. 

For $B\leq \sym_n$ and $\ell\in [n]$ denote by $\stir{B}{\ell}$ the number of
permutations in $B$ with precisely~$\ell$ cycles in their disjoint
cycle decomposition; we note that $\stir{\sym_n}{\ell}$ is the
Stirling number $\stir{n}{\ell}$ of the first kind,
see~\cite[Section~6.1]{GKP89} for properties of $\stir{n}{\ell}$. Our first result concerns  $\delta_k(A\wr_I B)$ and $\delta_k(A\power B)$.

\begin{theorem}\label{thm_formulas}
  If $m,n\ge2$, $A\le\sym_m$, and $B\leq \sym_n$, then
  \begin{align*}
    \delta_k(A\wr_I B) &= \sum_{\ell=0}^n \delta_\ell(B) \sum_{j_1+\dots+j_\ell=k} \prod_{r=1}^{\ell} \delta_{j_r}(A)
    \quad\text{and so}\quad
    \delta(A\wr_I B) = \sum_{\ell=0}^n \delta_\ell(B) \delta(A)^\ell;\\[1ex]
    \delta_k(A\power B) &= \frac{1}{|B|}\sum_{\ell=1}^n\stir{B}{\ell} \sum_{j_1\cdots j_\ell=k} \prod_{r=1}^\ell \delta_{j_r}(A)
  \text{\ and \ }
  \delta(A\power B)=1-\frac{1}{|B|}\sum_{\ell=1}^n \stir{B}{\ell}(1-\delta(A))^\ell.
    \end{align*}
\end{theorem}

The formulas for $\delta_k(A\imprim  B)$ and $\delta_k(A\power  B)$
are proved in
Theorems~\ref{thm_imprim_k} and \ref{thm_power}. Setting $k=1$ in Theorem~\ref{thm_formulas} gives the simple formulas for
$\delta_1(A\imprim B)$ and $\delta_1(A\power B)$ stated in
Corollaries~\ref{C:Iw} and~\ref{C:Pw}. {Formulas for $k=0$ were
known~\cite[Theorems 4.3, 5.4(1)]{Boston} and follow easily from our results.}

We now summarise our density results; as usual, for $a\le b$ we write $[a,b]=\{x\in\mathbb{R}\mid a\le x\le b\}$.

\begin{theorem}\label{thm_imprim1}
  For fixed $B\leq \sym_n$, the set
  $\{\delta(A\imprim B) \mid A\le\sym_m \text{ primitive},\; m\in\mathbb{N}\}$
  is dense in $[\delta(B),1]$.
  For fixed $A\leq \sym_m$, the set
  $\{\delta(A\imprim B)\mid B\le\sym_n\text{ imprimitive},\; n\in\mathbb{N}\}$
  is dense in $[\delta(A),1]$.  
\end{theorem}

\begin{theorem}\label{thm_power_dense1}
    For fixed $B\leq \sym_n$, the set
  $\{\delta(A\power B) \mid A\le\sym_m \text{ primitive},\; m\in\mathbb{N}\}$
  is dense in~$[0,1]$.
  For fixed $A\leq \sym_m$, the set 
  $\{\delta(A\wr_P B) \mid B\le\sym_n \text{ regular},\;n\in\mathbb{N}\}$
  is dense in $[\delta(A),1]$.
\end{theorem}

\begin{theorem}\label{thm_power_new1}
  Fix $n\ge2$ and $B\le\sym_n$. Let $(C_m)_{m\ge1}$ and $(A_m)_{m\ge1}$ be sequences of subgroups such that $C_m\lhdeq A_m\le\sym_m$. Let $(G_m)_{m\geq 1}$ be a sequence of subgroups such that $G_m\le\sym_{m^n}$ satisfies 
  $C_m^{[n]}\lhdeq G_m\le A_m\power B$, where $C_m^{[n]}\leq A_m^{[n]}$ is the subgroup of all functions $[n]\to C_m$, and $B=\pi_n(G)$ is the image of the projection map $\pi_n\colon \sym_m\power \sym_n\to \sym_n$. Suppose there exists $\delta_0$  such that $\lim_{m\to\infty} \delta(C_ma_m)=\delta_0$ for each sequence $(a_m)_{m\ge1}$ of elements $a_m\in A_m$. Then
  \[
   \lim_{m\to\infty}\delta(G_m)
   =1-\frac{1}{|B|}\sum_{\ell=1}^n\stir{B}{\ell}(1-\delta_0)^\ell.
  \]
\end{theorem}

Theorem~\ref{thm_imprim1} is proved in Section~\ref{sec_imprim}. Most of the paper
deals with power action in Section~\ref{sec:power} where
Theorems~\ref{thm_power_dense1} and~\ref{thm_power_new1} are proved.
{Theorem~\ref{thm_power}(b) is a generalisation of~\cite[Theorem~5.4(2)]{Boston} (see also Theorem~\ref{boston}(b) below). Derangements of `large' primitive subgroups $G\leq \sym_m\power\sym_n$, i.e. those satisfying $\alt_m^{[n]}\lhdeq G\le \sym_m\power\sym_n$, are considered in Corollary~\ref{C:large}.}

Further results of this paper include the determination of $\delta_k(G)$ for sharply $t$-transitive subgroups $G\le\sym_n$, see Section~\ref{sec:sharp}, and the determination of $\delta_k(C)$ for direct
products $C_1\times\dots\times C_r$ with intransitive or product actions,
see Theorems~\ref{thm_productNEW} and \ref{thm_intrans}.

\subsection{Motivation}\label{sec_motivation}
  The motivation that led us to Theorem~\ref{thm_power_new1} involved
  the study of primitive permutation groups that are `large', or
  are `diagonal'~\cite{LPS}. We explain the former. A \emph{base}
  for a permutation group $G\le\sym_n$ is a set $\{x_1,\dots,x_m\}$
  of points of $[n]$ such that the elementwise stabiliser
  $G_{(x_1,\dots,x_m)}$ is trivial.  The minimal size of a base for
  $G\le\sym_n$ is denoted $b(G,[n])$, or $b(G)$. It is clear that
  $|G|\le\prod_{i=0}^{b(G)-1}(n-i)$. In fact, the size of $b(G)$ is a
  good proxy for the size of $G$ for reasons that are related to
  Pyber's (now solved) base size conjecture, see~\cite{DHM}.  When
  considering permutation groups computationally or theoretically,
  groups with small base size (and hence small order) are treated
  very differently. For example, Seress describes very different
  algorithms for small and large base groups in~\cite{Seress}, and the
  `large' primitive groups with product action are considered
  separately by Maróti~\cite{Maroti}. In particular, primitive groups
  are frequently divided into two categories  (`small' and `large'), see
  \cite[Theorem~1.1]{Maroti}.
  The \emph{`small'} primitive subgroups
  $G\le\sym_N$  satisfy $|G| \leq  N^{1+{\lfloor\log_2(N)\rfloor}}$ or are one for four
  simple groups,  and the \emph{`large'} primitive
  groups satisfy $\alt_m^{[n]} \lhdeq G \le \sym_m\wr\sym_n$ where
  $\sym_m$ acts on $k$-subsets of $[m]$ and $N=\binom{m}{k}^n$.
  `Large' primitive groups arise when considering orders
  \cite[Theorem~1.1]{Maroti}, base sizes \cite[Theorem]{Liebeck84},
  and minimal degrees \cite[Theorem 2]{LiebeckSaxl} of primitive groups. {We note that the definition of `large' groups does not depend on the Classification of Finite Simple Groups, but the fact that they are almost always larger than `small' groups does.}

The `large' primitive permutation groups are important in many
computational and theoretical contexts. Given such a group, we
wondered whether knowing only $\delta(G)$ was enough to determine the
projection $\pi(G)\le\sym_n$ in Theorem~\ref{thm_power_new1}. In the
light of Theorem~\ref{thm_power_dense1} this may seem like an
impossible hope.  However, if $G=A\power B$ is the full wreath product, then
$\delta(G)=1-\mathcal{C}_B(1-\delta(A))$ holds
by~\cite[Theorem~5.4(2)]{Boston}. If $A=\sym_m$, then $\delta(G)$
strongly influences the polynomial $\mathcal{C}_B(x)$ when $n$ is
small, and we can sometimes recover the group~$B$.  (For example, if
$n<6$, then the polynomial $\mathcal{C}_B(x)$ determines the group
$B$, and if $B$ is primitive, then $\mathcal{C}_B(x)$ determines $B$
if $n<64$.)  By taking $C=\alt_m$, Theorem~\ref{thm_power_new1} gives
hope for recovering $\pi(G)$ from $\delta(G)$ even when $G$ is
\emph{not} a wreath product. In order to estimate $\delta(G)$
  when $\alt_m^{[n]}\lhdeq G\le \sym_m\power\sym_n$, we consider
  proportions of elements in \emph{subsets} of $G$.

Our interest was piqued by a recent
announcement~\cite[Theorem 1.1]{Poonen} 
that if
$G\leq\sym_n$ is a subgroup with $\delta(Gc)=\delta(\sym_n)$ for
some $c\in \sym_n$, then $Gc=\sym_n$ and hence $G=\sym_n$;
see~\cite[Remark 1.2]{Poonen} for a comment on subsets $C\leq \sym_n$
with $\delta(C)=\delta(\sym_n)$.
{On a computational note, we mention that Arvind~\cite{arvind} presents
  an algorithm which takes as input $k\in[n]$ and a subgroup
  $\langle S\rangle\le\sym_n$ and outputs a permutation in $\langle S\rangle$
  that moves at least~$k$ points: to find a derangement set $k=n$.}

\section{Background}\label{sec_background}
\noindent If $G\leq \sym_n$ is a transitive subgroup
and $n\geq 2$, then the Orbit-Counting Theorem implies that
$\Der(G)$ is non-empty; {this result is due to Jordan, see~\cite[Theorem 4]{Serre}.}
Given a transitive permutation group $G\leq \sym_n$, its \emph{rank} $r$  is the number of
orbits of a point-stabiliser on $[n]$. For such a $G$, 
the following bounds hold:
\begin{equation}\label{eq_DFG}
  \frac{r-1}{n}\leq \pder{G} \leq 1-\frac{1}{r}. 
\end{equation}
The lower bound was proved by {Cameron and Cohen~\cite{CC92}}, and the
upper bound by {Diaconis, Fulman, and Guralnick~\cite[Theorem 3.1]{DFG08}}, see {also Guralnick, Isaacs, and Spiga~\cite{guralnick} for a short proof.}

Serre~\cite{Serre} describes some interesting consequences of the
existence of derangements to number theory and topology. 
{Fulman and Guralnick~\cite{FG18}} show that
if $G$ is a sufficiently large finite simple group acting faithfully
and transitively on~$[n]$,
then $\delta(G)\geq 0.016$. This result completes the proof of
the Boston--Shalev Conjecture, which claims that there is a constant
$\varepsilon>0$ such that $\delta(G)>\varepsilon$ for any such group~$G$.

In the course of our research, we proved three (quite natural) results that
were previously proved by
Boston \emph{et al.}~\cite[Theorems 4.3, 5.4(2), 5.11]{Boston}.

\begin{theorem}[Boston \emph{et al.}\ \cite{Boston}]\label{boston}
  Let $A\leq \sym_m$ and $B\leq \sym_n$. Then
  \begin{enumerate}[{\rm (a)}]
  \item $\delta(A\imprim B)= \mathcal{P}_B(\delta(A))$ where
    $\mathcal{P}_B(x)=\sum_{\ell=0}^n \delta_\ell(B)x^\ell$.
    \item $\delta(A\power B)=1-\mathcal{C}_B(1-\delta(A))$ where
       $\mathcal{C}_B(x)= \frac{1}{|B|}\sum_{\ell=1}^n \stir{B}{\ell}x^\ell$.
    \item The set $\{\delta(C)\mid \text{$C\le\sym_n$ primitive},\; n\in\mathbb{N}\}$ is dense in $[0,1]$.
  \end{enumerate}
\end{theorem}

\subsection{Cycle index polynomials and derangements}\label{sec:cycleindex}
For $g\in\sym_n$ and $i\in[n]$ let $c_i(g)$ denote the number of $i$-cycles
in the disjoint cycle decomposition of~$g$, and {let $c(g)$ denote the number of cycles of $g$, so that $c(g)=\sum_{i=1}^n c_i(g)$}.
The well-known \emph{cycle index polynomial} $\mathcal{Z}(G)$ of $G\leq \sym_n$ is defined
to be the multivariate polynomial  
\[
\mathcal{Z}(G)=\mathcal{Z}(G;x_1,\dots,x_n) = \frac{1}{|G|}\sum_{g\in G} x_1^{c_1(g)}\cdots x_n^{c_n(g)}.
\]
Specialising the variables $x_i$ allows one to obtain polynomials with fewer variables which
are easier to calculate. For example, the  probability generating functions  $\mathcal{P}_G(x)$ and $\mathcal{C}_G(x)$
for the number of fixed points and the number of permutations with $\ell$ cycles, respectively, are
\begin{align*}
  \mathcal{P}_G(x)&=\mathcal{Z}(G;x,1,\dots,1)
  =\frac{1}{|G|}\sum_{\ell=0}^n |\Fix_\ell(G)| x^\ell
  =\sum_{\ell=0}^n \delta_\ell(G) x^\ell\qquad\textup{and}\\
  \mathcal{C}_G(x)&=\mathcal{Z}(G;x,x,\dots,x)
   =\frac{1}{|G|}\sum_{g\in G} x^{c(g)}
  =\frac{1}{|G|}\sum_{\ell=1}^n \stir{G}{\ell} x^\ell,
\end{align*}
where, as in Section \ref{sec_intro}, we denote
{the number of $g\in G$ with $c(g)=\ell$ by
  \begin{eqnarray}\label{eq_stir}\stir{G}{\ell}=\left|\{g\in G \mid c(g)=\ell\}\right|.
\end{eqnarray}}Clearly, $g\in\Fix_k(G)$ if and only if $c_1(g)=k$; in particular, $\delta(G)=\mathcal{P}_G(0)$. We note that for $k=0$, the formulas in Theorems \ref{thm_formulas} and~\ref{thm_power_new1} can be expressed using the cycle index polynomial.

Formulas for the cycle index polynomials $\mathcal{Z}(A\imprim B)$
and $\mathcal{Z}(A\power B)$ date back to {P\'{o}lya and to Palmer and Robinson,}
respectively \cite[Theorems 1, 2]{PR}. Our proof of Theorem~\ref{thm_formulas}
does not use these formulas, as they are rather complicated, especially
for $\mathcal{Z}(A\power B)$. The GAP package WPE~\cite{Rober} is a useful
research tool which, in particular, can compute $\mathcal{Z}(A\power B)$. {We note that the fixed point polynomials and cycle polynomials are also studied in \cite{semeraro,harden}. For example,
  $\mathcal{C}_{A\wr_I B}(x)=\mathcal{C}_B(\mathcal{C}_A(x))$ holds
  for all $A\leq \sym_m$ and $B\leq \sym_n$ by~\cite[Proposition~8]{semeraro}.
}

\subsection{Group actions}\label{sec_actions}
We denote a group $G$ acting on a set $X$ by $(G,X)$. The corresponding homomorphism $\varphi\colon G\to\Sym (X)$ allows us
to identify $G/\ker\varphi$ with the subgroup $H=\varphi(G)\leq
\sym_n$ where $|X|=n$. {Since
  \[
  \frac{|\{g\in G\mid \textup{$\varphi(g)$ is a derangement}\}|}{|G|}
  =
  \frac{|\{h\in H\mid \textup{$h$ is a derangement}\}|}{|H|}
  \]
we define $\delta(G,X)=\delta(\varphi(G))$.} Thus
we shall henceforth consider \emph{faithful} actions, and view $G$ as
a \emph{subgroup} of $\Sym(X)$.
Moreover, if $Y\to X$ is a
surjection of $G$-sets, $\delta(G,Y)\ge\delta(G,X)$ holds,
see~\cite[p.~3]{FG03}, which implies that  when considering lower
bounds for the proportion of derangements of a group (for any action),
one can restrict to primitive actions.

Given a group $A$ and a permutation group $B\le\Sym(Y)$
the permutational wreath product $A\wr_Y B$ is defined as the
split extension $A^Y\rtimes B$ where $A^Y$ is the group of all
functions $Y\to A$ with pointwise multiplication. The group
$A^Y\rtimes B$ has underlying set $A^{Y}\times B$ and multiplication
\[
(\alpha,b)(\beta,c)=(\alpha\beta^{b^{-1}},bc)
\qquad\textup{where $\alpha,\beta\in A^Y$, $b,c\in B$,}
\]
and $\beta^b\colon Y\to A$ is the map $y\mapsto\beta(yb^{-1})$.
Following {Cameron, Gewurz, and Merola~\cite{CGM08}}, given permutation
groups $A\le\Sym(X)$ and $B\le\Sym(Y)$ the
direct product $A\times B$ has an \emph{intransitive} action~\Ix\
and a \emph{product} action~\Px\, and the wreath product $A\wr_Y B$ has
an \emph{imprimitive} action~\Iw\ and a \emph{power} action~\Pw,
defined as follows. To avoid towers of exponents, we write the action
of $g\in G$ on $x\in X$ as $xg$, and not $x^g$. Also $X\,\dot\cup\, Y$
denotes the disjoint union of $X$ and $Y$.

\begin{enumerate}
\item[\Ix]$(a,b)\in A\times B$ acts on $X\,\dot\cup\, Y$
  via $x(a,b)=xa$ if $x\in X$, and $y(a,b)=yb$ if $y\in Y$.
\item[\Px] $(a,b)\in A\times B$ acts on  $(x,y)\in X\times Y$ via
  $(x,y)(a,b) \defi (xa,yb)$.
\item[\Iw] $(\alpha,b)\in A\wr_Y B$ acts on $(x,y)\in X\times Y$ via
  $(x,y)(\alpha,b)\defi (x\alpha(y),yb)$.
\item[\Pw] $(\alpha,b)\in A\wr_Y B$ acts on $\omega\in X^Y$ via
 $\omega(\alpha,b)\colon Y\to X$,  $y\mapsto \omega(yb^{-1})\alpha(yb^{-1})$.
\end{enumerate}

In the case that $Y=[n]$, we identify  $\omega\in X^Y$ and $\alpha\in A^Y$
with $n$-tuples $(\omega_1,\dots,\omega_n)$, $(\alpha_1,\dots,\alpha_n)$
where each $\omega_i=\omega(i)$ and $\alpha_i=\alpha(i)$, respectively.
 The power action \Pw\ of the
base group $A^Y$ coincides with (iterated) product action \Px. Most authors call \Pw\ product action; although unconventional, we shall refer to \Pw\ as power action {to avoid confusion with \Px.} If $A\leq \Sym(X)$ is primitive but not regular, $B\leq \Sym(Y)$ is transitive, $|Y|$ is finite, and $1<|X|,|Y|$, then  $A\wr_Y B\le\Sym(X^Y)$
is primitive by~\cite[Lemma~2.7A]{DM}, so \Pw\ can suggest a
\emph{primitive}/\emph{product}/\emph{power wreath} product, so \Pw\ is a good abbreviation.
We change the subscript in $A\wr_Y B$ in favour of the
notation $\imprim$ or~$\power$.

\begin{notation} We abbreviate $(A\times B,X\,\dot\cup\, Y)$,
  $(A \times B,X\times Y)$, $(A \wr_Y B,X\times Y)$, and  $(A\wr_Y B,X^Y)$
  by $A\intrans B$,\; $A\product B$,\; $A\imprim B$, and $A\power B$,
  respectively.
\end{notation}

 We emphasise that there is a permutation isomorphism
  $(A\imprim B)\imprim C\cong A\imprim (B\imprim C)$, but associativity
  fails for $\power$. Indeed $|(X^Y)^Z|=|X^{Y\times Z}|\ne |X^{(Y^Z)}|$
  if $|Y|\ge2$, $|Z|\ge2$ and $|Y|^{|Z|}\ne 4$.

\section{Sharply transitive groups}\label{sec:sharp}

\noindent Let $G\leq \sym_n$ be a subgroup that acts sharply $t$-transitively, that is,
it acts regularly on the set of $t$-tuples with distinct entries in $[n]$.
Hence $|G|$ equals the number $n!/(n-t)!$ of such~tuples. 

\begin{theorem}\label{AnSnNEW}
  Let $n\geq 2$, $t\in [n]$, and $k\in\{0,\ldots,n\}$. If $G\leq \sym_n$
  is sharply $t$-transitive, then $\delta_n(G)=1/|G|$,
  $\delta_k(G)=0$ if $t\leq k<n$, and
  \[
    \delta_k(G)
    = \frac{1}{k!}\sum_{j=0}^{t-k-1}\frac{(-1)^j}{j!}
    +\frac{(n-t)!}{k!}\sum_{j={t-k}}^{n-k} \frac{(-1)^{j}}{(n-k-j)!j!}
    \qquad\text{if $0\le k<t$.}
  \]  
\end{theorem}
\begin{proof}
  Note that $|\Fix_n(G)|=1$ and $|\Fix_k(G)|=0$ for $t\leq k<n$, so
  the claim is true for $k\geq t$.
  Suppose now that $k<t$.  For a subset $K\subseteq [n]$ let
  $G_{(K)}$ be the elementwise stabiliser of $K$ in $G$.
  Note that $|\Der(G,[n])|=|G|-|{F}|$, where ${F}=G_{(1)}\cup\dots\cup G_{(n)}$. This number can be determined using a standard inclusion-exclusion argument (see also \cite[Theorem 2.3]{Boston}), and we obtain
  \begin{equation}\label{E:B}
  \delta(G,[n])=\sum_{j=0}^{t-1}\frac{(-1)^j}{j!}
    +(n-t)!\sum_{j=t}^{n}\frac{(-1)^j}{j!(n-j)!}.
  \end{equation}
  The set $K$ of fixed points of $g\in\Fix_k(G)$ has size $k$, and $g$
  induces a derangement on the complement $K'=[n]\setminus K$. We view
  $G_{(K)}$ as a sharply $(t-k)$-transitive permutation group of
  degree $|K'|=n-k$. Since $|G_{(K)}|=(n-k)!/(n-t)!$, we have
  \begin{equation}\label{E:B2}
  |\Fix_k(G)|=\binom{n}{k}|G_{(K)}|\,\delta(G_{(K)},K')
  =\frac{n!}{k!(n-t)!}\,\delta(G_{(K)},K')
  =\frac{|G|}{k!}\,\delta(G_{(K)},K').
  \end{equation}
  Dividing by $|G|$, and using~\eqref{E:B} with $n$ and $t$ replaced by $n-k$
  and $t-k$ proves the claim.
\end{proof}

\begin{remark}
  For an alternative proof of Theorem~\ref{AnSnNEW}, let
  $\mathcal{Q}_G(x)=\sum_{\ell=0}^n a_\ell x^\ell/\ell!$ where 
  $a_\ell$ denotes the number of orbits of $G$ on $\ell$-tuples of
  distinct points. If $G$ is sharply $t$-transitive, then
  $a_\ell=\prod_{i=1}^{\ell-t}(n-\ell+i)$ if $t<\ell\le n$, and $a_\ell=1$
  if $\ell\le t$. We may deduce $\delta_k(G)$ from the
  polynomial identity $\mathcal{P}_G(x)=\sum_{k=0}^n\delta_k(G)x^k=\mathcal{Q}_G(x-1)$
  see~\cite{Boston} or~\cite[Theorem~1.1]{Cameron}.
\end{remark}

Throughout, we define $d_0=e_0=1$,
and if $n\geq 1$ is an integer, then 
\[
d_n\defi\sum_{k=0}^{n} \frac{(-1)^k}{k!}\quad\text{and}\quad
e_n\defi\sum_{k=n-1}^{n} \frac{(-1)^k}{k!}=\frac{(-1)^{n-1}(n-1)}{n!}.
\]
Note that  $\lim_{n\to\infty}e_n=0$ and $\lim_{n\to\infty} d_n = e^{-1}$,
where $e$ denotes the Euler number. Also, if $n\ge2$,
then $\sym_n$ is sharply $(n-1)$-transitive, and $\alt_n$ is sharply
$(n-2)$-transitive. The next result follows from
Theorem \ref{AnSnNEW} with a little algebra.
  
\begin{corollary}\label{OLDAnSn}
  If $n\geq 1$ and $k\in\{0,\dots,n\}$, then
  \begin{enumerate}[{\rm(a)}]
    \item $\delta_k(\sym_n)=d_{n-k}/k!$,
    \item $\delta_k(\alt_n)=(d_{n-k}+e_{n-k})/k!$, and 
    \item $\delta_k(\sym_n\setminus \alt_n)=(d_{n-k}-e_{n-k})/k!$.
  \end{enumerate} 
 If $k$ is fixed, then  $\lim_{n\to\infty}\pderk{k}{\sym_n}=\lim_{n\to\infty}\pderk{k}{\alt_n}
  =\lim_{n\to\infty}\pderk{k}{\sym_n\setminus \alt_n}=e^{-1}/k!$.
\end{corollary}

\begin{remark}
  If $G\leq \sym_n$ has $s$ orbits on $[n]$, then the Orbit-Counting Lemma
  implies that 
  \[
  \frac{1}{s}\sum_{k=0}^nk|\Fix_k(G)|=|G|=\sum_{k=0}^n|\Fix_k(G)|
  \quad\textup{and hence}\quad \delta(G)=  \sum_{k=1}^n \left(\frac{k}{s}-1\right)\delta_k(G).
  \]
  If  $(G,[n])$ is a sharply $t$-transitive subgroup of $\sym_n$,
  then $s=1$ and $|G|=n!/(n-t)!$, and Equation~\eqref{E:B2} transforms
  the latter equation for $\delta(G)$ into the following recurrence relation
 \[
  \delta(G,[n])=\frac{(n-t)!}{n(n-2)!}
  +\sum_{k=2}^{t-1}\frac{1}{k(k-2)!}\delta(G_{([k])},[n]\setminus[k]),
  \]
where $(G_{([k])},[n]\setminus[k])$ is a $(t-k)$-transitive subgroup of $\textup{Sym}([n]\setminus[k])\cong\sym_{n-k}$.
\end{remark}


\section{Direct products}\label{sec:2}

\subsection{Product action}\label{sec_product}
The following result generalises~\cite[Lemma~6.1]{Boston}.

\begin{theorem}\label{thm_productNEW}
 For $r>0$ and $i\in[r]$ let $G_i\leq \Sym(X_i)$. Let the subgroup
  $G=G_1\times\dots\times G_r$ of $\Sym(X_1\times\dots\times X_r)$
 act via \Px. {If $C\subseteq G$ is a subset of the
 form $C=C_1\times\dots\times C_r$ with
 each $C_i\subseteq G_i$}, then
 \[\pderk{k}{C}=\sum_{i_1\cdots i_r=k} \prod_{s=1}^r \delta_{i_s}(C_s)
 \quad\text{and}\quad
 \pder{C}=1-\prod_{i=1}^r(1-\delta(C_i))\quad\text{if $k=0$}.
 \]
\end{theorem}

\begin{proof}
  Note that $C$ is partitioned into sets $D_{i_1,\dots,i_r}={\rm Fix}_{i_1}(C_1)\times\dots\times {\rm Fix}_{i_r}(C_r)$ for $i_1,\dots,i_r\geq 0$, and every element in  $D_{i_1,\dots,i_r}$ has exactly $i_1\cdots i_r$ fixed points. Thus we have
  \begin{align*}
    |{\rm Fix}_k(C)| &= \sum_{i_1\cdots i_r=k} \prod_{s=1}^r |{\rm Fix}_{i_s}(C_s)|
  \end{align*}
  and dividing by $|C|=\prod_{s=1}^r |C_i|$ yields the first claim.

  For $k=0$, note that  $(x_1,\dots,x_r)\in X_1\times\cdots\times X_r$ is fixed
 by $(c_1,\dots,c_r)\in C$ if and only if $x_ic_i=x_i$ for each $i\in[r]$. Thus, $1-\delta(C)=\prod_{i=1}^r(1-\delta(C_i))$, as claimed.
\end{proof}

We now give examples of subgroups all having similar
proportions of derangements. 

\begin{corollary}\label{cor_product2}
  Let $G_{m_1,\dots,m_r}\le \sym_{m_1}\times\dots\times \sym_{m_r}$ act via
  product action \Px\ where
  \[\alt_{m_1}\times\dots\times \alt_{m_r} \lhdeq G_{m_1,\dots,m_r}\leq \sym_{m_1}\times\dots\times \sym_{m_r}.\]
  For the multi-indexed sequence $(G_{m_1,\dots,m_r})$ of subgroups, we have
  \[\lim_{m_1,\dots,m_r\to\infty} \pder{G_{m_1,\dots,m_r}} =1-(1-e^{-1})^r.\]
\end{corollary}
\begin{proof}
  Note that $G=G_{m_1,\dots,m_r}$ is a union of cosets of
  $A=\alt_{m_1}\times \dots \times \alt_{m_r}$ of the
  form
  \[
  C=A(c_1,\dots,c_r)=(\alt_{m_1}c_1)\times \dots \times (\alt_{m_r}c_r).
  \]
  For each $i$, either $\alt_{m_i}c_i=\alt_{m_i}$ or
  $\alt_{m_i}c_i=\sym_{m_i}\setminus \alt_{m_i}$, so $\pder{\alt_{m_i}c_i}=d_{m_i}\pm e_{m_i}$
  by Corollary~\ref{OLDAnSn}. Since each $\pder{\alt_{m_i}c_i}\to
  e^{-1}$ for $m_i\to\infty$, Theorem~\ref{thm_productNEW} implies
  that $\pder{Ac}\to 1-(1-e^{-1})^r$ as $m_1,\dots,m_r\to \infty$.
  If $G$ is a disjoint union $G=\bigcup_{c\in\mathscr{C}}Ac$, so $|G|=|A||\mathcal{C}|$, then
  \[\pder{G} = \frac{1}{|G|}\sum_{c\in\mathscr{C}}|\Der(Ac)|=\frac{1}{|\mathscr{C}|} \sum_{c\in\mathscr{C}}\delta(Ac).\]
 As shown above, $\sum_{c\in\mathscr{C}}\delta(Ac)$ converges to
 $|\mathscr{C}|(1-(1-e^{-1})^r)$, which implies the claim.
\end{proof}

\subsection{Intransitive action}\label{sec_intrans}

Counting derangements with an intransitive action is straightforward.

\begin{theorem}\label{thm_intrans}
  For $r>0$ and $i\in[r]$ let $G_i\leq \Sym(X_i)$. Let the subgroup $G=G_1\times\dots\times G_r$ of $\Sym(X_1\,\dot{\cup}\,\cdots\,\dot{\cup}\,X_r)$ act via \Ix. 
  If $C\subseteq G$ {is a subset of the form
  $C=C_1\times\dots\times C_r$ with each $C_i\subseteq G_i$}, then
  \[
  \pderk{k}{C}=\sum_{k_1+\dots+k_r=k}\prod_{i=1}^r \pderk{k_i}{C_i}
  \qquad\text{and hence}\qquad
  \delta(C)=\prod_{i=1}^r \delta(C_i).
  \]
\end{theorem}
 
\begin{proof}
  We have $(c_1,\dots,c_r)\in{\rm Fix}_k(C)$ if and only if
  for each $i\in[r]$, there exists a $k_i$ such that $c_i\in{\rm Fix}_{k_i}(C_i)$
  and
  $k=k_1+\cdots + k_r$. The result follows from
  \[|{\rm Fix}_k(C)|=\sum_{k_1+\dots+k_r=k}\prod_{i=1}^r |{\rm Fix}_{k_i}(C_i)|.\qedhere\]
\end{proof}

The following result gives subgroups with similar proportions of
derangements. We omit the proof as it similar to that of
Corollary \ref{cor_product2} but using Theorem~\ref{thm_intrans} instead
of Theorem~\ref{thm_productNEW}.

\begin{corollary}\label{cor_intrans}
  Let $G_{m_1,\dots,m_r}\le \sym_{m_1}\times\dots\times \sym_{m_r}$ act
  via intransitive action \Ix\ where
  \[
  \alt_{m_1}\times\dots\times \alt_{m_r} \lhdeq G_{m_1,\dots,m_r}
  \leq \sym_{m_1}\times\dots\times \sym_{m_r}.
  \]
  The multi-indexed sequence $(G_{m_1,\dots,m_r})$ satisfies
  $\lim_{m_1,\dots,m_r\to\infty} \pder{G_{m_1,\dots,m_r}} = e^{-r}$.
\end{corollary}


\section{Wreath products \texorpdfstring{$A\imprim B$}{} with imprimitive action}\label{sec_imprim}
\noindent In this section we find a formula for $\delta_k(A\wr_I B)$, which is the first formula
in Theorem~\ref{thm_formulas}, and we prove the
density result Theorem~\ref{thm_imprim1}.
 
  \begin{theorem}\label{thm_imprim_k}
    Let $A\leq \sym_m$ and $B\leq \sym_n$. Then
    \[
    \delta_k(A\wr_I B) = \sum_{\ell=0}^n \delta_\ell(B) \sum_{j_1+\dots+j_\ell=k} \prod_{r=1}^{\ell} \delta_{j_r}(A)
    \quad\textup{and hence}\quad
    \delta(A\wr_I B) = \sum_{\ell=0}^n \delta_\ell(B) \delta(A)^\ell.
    \]
  \end{theorem}
 \begin{proof} 
An element $(\alpha,b)\in A\wr_I B$ fixes $(x,y)\in[m]\times[n]$ if
and only if $x=x\alpha(y)$ and $y=yb$.  If $y_1,\dots,y_\ell$ are the
fixed points of $b$ (so $b\in {\rm Fix}_\ell(B)$) and
$x_{i,1},\ldots,x_{i,j_i}$ are the fixed points of $\alpha(y_i)$ (so
$\alpha(y_i)\in {\rm Fix}_{j_i}(A)$), then each $(x_{i,s},y_i)$ is a
fixed point of $(\alpha,b)$, and every fixed point of $(\alpha,b)$ has
this form. No constraints are imposed upon
$\alpha(y)$ for $y\in[n]\setminus\{j_1,\dots,j_\ell\}$.
The total number of fixed points is $k=j_1+\dots+j_\ell$, hence
$(\alpha,b)\in {\rm Fix}_{k}(A\wr_I B)$. This shows the following, and
dividing by $|A\wr_I B|=|A|^n|B|$ proves the claim:
\[
|{\rm Fix}_k(A\wr_I B)| = \sum_{\ell=0}^n |{\rm Fix}_\ell(B)| \sum_{j_1+\dots+j_\ell=k} |A|^{n-\ell}\prod_{r=1}^{\ell} |{\rm Fix}_{j_r}(A)|.\qedhere
\]
 \end{proof}

\begin{corollary}\label{C:Iw}
  If $A\le\sym_m$ and $B\le\sym_n$, then
  \[\delta_1(A\imprim B)=\delta_1(A)\mathcal{P}_B'(\delta(A))\]
  where $\mathcal{P}_B'(x)=\sum_{\ell=1}^n\ell\delta_\ell(B)x^{\ell-1}$
  is the derivative of $\mathcal{P}_B(x)$.
\end{corollary}

The polynomial $\mathcal{P}_B(x)$
{can be difficult to compute precisely, but in some cases can be
easy to approximate. For example},
  when $B$ is transitive and $\delta(B)$ is small, then
  $\mathcal{P}_B(x)$ is convex  for $x\in[0,1]$, that is, cup shaped,
  and slightly larger than $x$ by the following~lemma. 
  
\begin{lemma}\label{L:convex}
  If $B\le\sym_n$, then $\mathcal{P}_B(x)$ is convex.
  If $B$ is transitive, then
    $0\le \mathcal{P}_B(x)- x\le \mathcal{P}_B(0)$ holds for $0\le x\leq 1$.
\end{lemma}
\begin{proof}
  The claim is trivially true if $n=1$, so let $n\geq 2$ and write $f(x)=\mathcal{P}_B(x)-x$. Since $\mathcal{P}_B(x)$ has non-negative coefficients and highest term $x^n/|B|$,  the second  derivative satisfies $f''(x)=\mathcal{P}''_B(x)>0$ for $0<x\leq 1$, so both $f(x)$ and $\mathcal{P}_B(x)$ are convex on $[0,1]$. By  \cite[Theorem~4.6(4)]{Boston}, we have $\mathcal{P}'_B(x)=\mathcal{P}_H(x)$ where $H\leq \sym_{n-1}$ is the stabiliser of $n$ in $B$ acting on $[n-1]$. This implies $\mathcal{P}'_B(1)=1$ and $\mathcal{P}'_B(x)\leq 1$ for $x\in[0,1]$. Thus, $f(1)=f'(1)=0$, and convexity implies that $f(x)$ lies above the tangent line $y=0$ at $(1,0)$, that is, $0\leq f(x)$ on the domain $[0,1]$. Since $f'(x)=\mathcal{P}_B'(x)-1\leq 0$ on $[0,1]$, the function $f(x)$ is decreasing on $[0,1]$. This proves the last claim $f(x)\leq f(0)=\mathcal{P}_B(0)$.
\end{proof}

We now prove Theorem \ref{thm_imprim1}, employing a similar argument as in \cite[Theorem 4.13]{Boston}. 

\begin{proof}[Proof of Theorem \ref{thm_imprim1}]
  Fix $B\le\sym_n$. The map $\mathcal{P}_B\colon [0,1] \to (0,\infty)\colon
  z\mapsto\sum_{k=0}^n\delta_k(B)z^k$ satisfies
  $\delta(A\wr_I B)=\mathcal{P}_B(\delta(A))$ by Theorem~\ref{boston}(a).
  Note that $\mathcal{P}_B$ is a continuous increasing polynomial function
  with $\mathcal{P}_B(0)=\delta(B)$ and $\mathcal{P}_B(1)=1$, so
  $\delta(B)\leq\delta(A\power B)<1$
  for all  $A\le\sym_m$. Letting $m\ge1$ and $A$
  vary over the  primitive subgroups of $\sym_m$,
  it follows from Theorem~\ref{boston}(c) that the
  values of $\delta(A\power B)$ are dense in the interval $[\delta(B),1]$.
  This proves the first claim.

  For the second claim, fix $A\leq \sym_m$.
  If $B\leq \sym_n$ is imprimitive, then $B$ is transitive and $\delta(A\imprim B)=\mathcal{P}_B(\delta(A))\ge\delta(A)$
  by Lemma~\ref{L:convex}, so the values of $\delta(A\imprim B)$ (with $B$ imprimitive) lie
  in $[\delta(A),1]$.
  Let $q$ be a prime power,
  and let $C_q=\textup{AGL}_1(q)\le\sym_q$ act naturally on $q$ affine points;
  we also abbreviate $C=C_q$.   An easy calculation shows that
  {\[\mathcal{P}_C(x)=\frac{1}{q}+\frac{(q-2)x}{q-1}+\frac{x^q}{q(q-1)},\] see~\cite[Example~4.4]{Boston}}.
  Let $B_r=B_{r,q}$ be the  imprimitive wreath
  product $C\imprim\cdots\imprim C$ of  $r$ copies of $C$.
  We show that the  closure of the set
  $\{\delta(A\imprim B_{r,q})\mid \textup{$r\ge0$, $q$ prime power}\}$
  is the interval $[\delta(A),1]$.
  
  The recurrence $\mathcal{P}_{B_r}(x)=\mathcal{P}_{B_{r-1}}(\mathcal{P}_C(x))$
  holds by~\cite[Theorem~4.6(8)]{Boston} since $B_r=C\imprim B_{r-1}$.
  The function $\mathcal{P}_{B_{r-1}}(x)$
  is  increasing, and since $\mathcal{P}_C(x)>\alpha+\beta x$ for $\alpha=\frac{1}{q}$
  and $\beta=\frac{q-2}{q-1}$, we have
  $\mathcal{P}_{B_r}(x)> \mathcal{P}_{B_{r-1}}(\alpha+\beta x)$. An induction
  on $r$ now shows that 
  $\mathcal{P}_{B_r}(x)> \alpha(1+\beta+\cdots+\beta^{r-1})+\beta^r x$, and therefore $\mathcal{P}_{B_r}(x)>\alpha\frac{\beta^r-1}{\beta-1}
  =(1-\frac{1}{q})(1-\beta^r)$ for all $x\in[0,1]$.
  Now choose $q>1/\varepsilon$, so that
  $1-\frac{1}{q}>1-\varepsilon$, and then choose $r$ such that
  $(1-\frac{1}{q})(1-\beta^r)>1-\varepsilon$.

  Let $c_r=\delta(A\imprim B_r)$.
  It follows from Theorem~\ref{boston}(a) that
  $\delta(A\imprim B_r)=\mathcal{P}_{B_r}(\delta(A))$ and
  the previous paragraph shows that  $c_r>1-\varepsilon$.
    As $C$ is transitive, $0\le\mathcal{P}_C(x)-x\le\mathcal{P}_C(0)$ for
    $x\in [0,1]$ by Lemma~\ref{L:convex},
  and so $0\le \mathcal{P}_C(x)-x\le\frac{1}{q}<\varepsilon$.
  Set $c_0=\delta(A)$ and $c_i=\mathcal{P}_C(c_{i-1})$. Then
  $0\le c_i-c_{i-1}<\varepsilon$ holds for
  $1\le i\le r$ and $1-\varepsilon<c_r<1$. The claim now follows since
  each $c_i=\delta(A\wr_I C\wr_I\cdots \wr_I C)=\delta(A\wr_I B_{i,q})$
  with $B_{i,q}$ imprimitive.
\end{proof}

In the special case that $B=\sym_m$, we obtain the following explicit formula.

\begin{corollary}
{Setting $B=\sym_n$} in Theorem~\ref{thm_imprim_k} gives
  \[
  \delta(A \imprim \sym_n)
  = \sum_{\ell=0}^n \frac{\delta(A)^\ell}{\ell!} \sum_{i=0}^{n-\ell} \frac{(-1)^i}{i!}
  = \sum_{i=0}^n \frac{(-1)^i}{i!} \sum_{\ell=0}^{n-i} \frac{\delta(A)^\ell}{\ell!}.
  \]
  Hence $\lim_{n\to\infty} \pder{A \imprim \sym_n} = e^{\delta(A)-1}$, and
  so  $\lim_{n,m\to\infty} \pder{\sym_m \imprim \sym_n} = e^{e^{-1}-1}$.
\end{corollary}
\begin{proof}
  The first displayed formula follows from Theorem~\ref{thm_imprim_k} since
  $|\delta_\ell(\sym_n)| = d_{n-\ell}/\ell!$ by Corollary~\ref{OLDAnSn}. Interchanging
  summations gives the second formula. The first limit now follows,
  and the second follows from $\lim_{m\to\infty}\delta(\sym_m)=e^{-1}$ by Corollary~\ref{OLDAnSn}.
\end{proof}

\section{Wreath products \texorpdfstring{$A\power B$}{} with power action \texorpdfstring{\Pw}{}}\label{sec:power}
\noindent Given $A\leq \sym_m$ and $B\leq \sym_n$, we study $A\power B$ acting with
power action on the set $\Omega=[m]^{[n]}$. Recall that the elements of
$A\power B$ are $(\alpha,b)$ with $\alpha\in A^{[n]}$ and $b\in  B$, and that we  sometimes write $\alpha\in A^{[n]}$ and $\omega\in \Omega$ as
$\alpha=(\alpha_1,\dots,\alpha_n)$ and $\omega=(\omega_1,\dots,\omega_n)$,
respectively. Using \Pw, the element $(\alpha,b)\in A\power B$ acts on
$\omega\in\Omega$ via
\[
\omega{(\alpha,b)}\colon [n]\to [m],\quad  y\mapsto \omega(yb^{-1}){\alpha(y{b^{-1}})}.
\] 

We start with a brief discussion of the numbers $\stir{G}{\ell}$ defined in {Equation \eqref{eq_stir}.}  If $G=\sym_n$, then $\stir{G}{\ell}=\stir{n}{\ell}$ is the  
(unsigned) Stirling number $\stir{n}{\ell}$ of the first kind, see
\cite[Section 6.1]{GKP89} for details. It is easy to see that for
$n\ge1$ and $G\leq \sym_n$ we have
\[
\stir{n}{1}=(n-1)!,\quad\stir{n}{n-1}=\binom{n}{2},
\quad \stir{n}{n}=1\quad\text{and}\quad
\sum_{\ell=0}^n\stir{G}{\ell}=|G|.
\]

\begin{example}\label{example:stir}
  If $C_n=\langle(1,2,\dots,n)\rangle\le\sym_n$, then 
  $\stir{C_n}{d}=\phi(n/d)$ if $d\mid n$, and $0$ otherwise, where
  $\phi$ denotes Euler's $\phi$-function, see \cite[Lemma 5.6]{Boston}: 
  The disjoint cycle decomposition of $(1,2,\dots,n)^k$ consists of $n/d$
  cycles each of length $d=\gcd(n,k)$. Such an element has order $n/d$, and there are $\phi(n/d)$ such elements in $B$.
   In particular, $\sum_{d\mid n}\phi(n/d)=n$.    
\end{example}

For $\alt_n$, we observe an alternating behaviour:
$\stir{\alt_n}{\ell}=0$ if $\ell\not\equiv n\bmod 2$,
and $\stir{\alt_n}{\ell}=\stir{n}{\ell}$ otherwise.
This follows from the next lemma.

\begin{lemma}
  If $G\leq \sym_n$, then $\stir{G}{\ell}=\stir{G\,\cap\, \alt_n}{\ell}$
  if $\ell\equiv n\bmod 2$, and
  $\stir{G}{\ell}=\stir{G\,\cap\, (\sym_n\setminus \alt_n)}{\ell}$ otherwise.
\end{lemma}
\begin{proof}
  Let $g\in G$ have cycle decomposition $g=g_1\cdots g_\ell$. Note that $|g_1|+\dots+|g_\ell|=n$, and $g_i$ has sign $-1$ if and only if $|g_i|$ is even. We can assume $g_1,\dots,g_s$ have sign $-1$ and $g_{s+1},\dots,g_\ell$ have sign~$1$, so $g\in \alt_n$ if and only if $s$ is even. The cycles $g_1,\dots,g_s$ involve an even number $e=|g_1|+\dots+|g_s|$ of points, and the remaining $\ell-s$ cycles involve  $n-e$ points. Each of $g_{s+1},\dots,g_\ell$ has odd length, which forces  $n-e\equiv \ell-s\bmod 2$. Thus, if $n\equiv \ell\bmod 2$, then $s\equiv e\equiv 0\bmod 2$, and $g\in \alt_n$. If $n\not\equiv \ell\bmod 2$, then $s\not\equiv e\equiv 0\bmod 2$, and $g\in \sym_n\setminus \alt_n$. The claim now follows.
\end{proof}
  
In particular, if $G\leq \sym_n$ but $\alt_n\not\leq G$, then $H=G\cap \alt_n$ has index $2$ in $G$, which implies that half of all elements lie in $H$ (and these elements all have cycle number congruent to $n$ modulo $2$) and the other half of elements lie in $G\setminus H$ (and {these elements all} have cycle number congruent to $n+1$ modulo $2$). This is summarised in the following lemma.
\begin{lemma}
  If $G\le\sym_n$ and $\alt_n\not\le G$, then
  $\displaystyle
\sum_{i=1}^{\lfloor n/2\rfloor} \stir{G}{2i}= \frac{|G|}{2} = \sum_{i=1}^{\lceil n/2\rceil} \stir{G}{2i-1}$.
\end{lemma}


\subsection{Formula for \texorpdfstring{$\delta_k(A\power B)$}{} in
  Theorem~\ref{thm_formulas}}\label{sec_power}
We start with a definition.

\begin{definition}\label{def:biproduct}
  Let $b\in \sym_n$ have disjoint cycle decomposition $b=b_1\cdots b_\ell$
  (including trivial cycles) and let $b_i=(y_i,y_ib,\dots,y_ib^{k_i-1})$ be
  a $k_i$-cycle whose smallest element is $y_i$.
  If $\alpha \in A^{[n]}$, then the \emph{$b_i$-product} of $\alpha$ is
  $b_i(\alpha)=\alpha({y_i})\alpha(y_ib)\cdots \alpha(y_ib^{k_i-1})\in A$.
\end{definition}

  \begin{theorem}\label{thm_power}
    Let $G=A\power B\leq \sym_{m^n}$ where  $A\leq \sym_m$ and $B\leq \sym_n$.
\begin{enumerate}[\rm (a)]
\item  Let $b\in B$ with disjoint cycle decomposition
  $b=b_1\cdots b_\ell$. Then
  \[
  \delta_k(A^{[n]}b) = \sum_{j_1\cdots j_\ell=k} \prod_{r=1}^\ell \delta_{j_r}(A);
  \]
  if $k=0$, then $\delta(A^{[n]}b)=1-(1-\delta(A))^\ell$.
  \item We have
    \[\delta_k(A\power B) = \frac{1}{|B|}\sum_{\ell=1}^n\stir{B}{\ell} \sum_{j_1\cdots j_\ell=k} \prod_{r=1}^\ell \delta_{j_r}(A);\]
    if $k=0$, then \[\delta(A\power B)=1-\frac{1}{|B|}\sum_{\ell=1}^n \stir{B}{\ell}(1-\delta(A))^\ell.\]
\end{enumerate}
  \end{theorem}
\begin{proof}
 (a) Suppose each cycle $b_i$ is defined as
 $b_i=(y_i,y_ib,\dots,y_ib^{k_i-1})$
  where $y_i$ is the smallest
  element in the cycle.  Suppose $(\alpha,b)\in G$ fixes some
  $\omega\in\Omega$.  Then $\omega(yb^{-1})\alpha(yb^{-1})=\omega(y)$
  for all $y\in[n]$ by~\Pw.  Letting $y=y_ib^{j+1}$ shows that for
  each $i$ and $j$ we have
  \[
  \omega(y_ib^j)\alpha(y_ib^j)=\omega(y_ib^{j+1});
  \]
  in particular,
  $\omega(y_i)\alpha(y_i)\alpha(y_ib)\cdots\alpha(y_ib^{k_i-1})=\omega(y_i)$,
  and so $\omega(y_i)$ is a fixed point of $b_i(\alpha)$. Moreover,
  the converse is true, and  $\omega$ is a fixed
  point of $(\alpha,b)$ if and only if each $\omega(y_i)$ is a fixed
  point of $b_i(\alpha)$, and $\omega(y_ib^{j+1})$ is defined as
  $\omega(y_i)\alpha(y_i)\alpha(y_ib)\cdots\alpha(y_ib^j)$.  If the
  fixed points of $b_i(\alpha)$ are $x_{i,1},\ldots,x_{i,j_i}$ (so
  $b_i(\alpha)\in {\rm Fix}_{j_i}(A)$), then $(\alpha,b)$ has exactly
  $j_1\cdots j_\ell$ fixed points. Thus, for $b$ as above, the number
  of $\alpha\in A^{[n]}$
  for which $(\alpha,b)$ has exactly $k$ fixed points is
\[
|{\rm Fix}_k(A^{[n]}b)|=\sum_{j_1\cdots j_\ell=k} |\{\alpha \in A^{[n]}
\mid \text{ each }b_i(\alpha)=\alpha(y_i)\alpha(y_ib)\cdots
\alpha(y_ib^{k_i-1})\in {\rm Fix}_{j_i}(A)\} |
\]
For each $i$, we have $b_i(\alpha)\in{\rm Fix}_{j_i}(A)$ if and only
if $\alpha(y_i)=c (\alpha(y_ib)\cdots\alpha(y_ib^{k_i-1}))^{-1}$ where
$c\in {\rm Fix}_{j_i}(A)$ and $\alpha(y_ib),\dots,
\alpha(y_ib^{k_i-1})\in A$ are arbitrary. Thus,
\[ |{\rm Fix}_k(A^{[n]}b)|= \sum_{j_1\cdots j_\ell=k} |A|^{n-\ell}\prod_{r=1}^\ell |{\rm Fix}_{j_r}(A)|,\]
and dividing by $|A^{[n]}b|=|A|^n$ proves the first claim.

The claim
for $k=0$ follows via induction on $\ell$. Indeed, if $\ell=1$,
then the claim is true by the above display. For $\ell>1$, we
split the sum into the cases $j_1=0$ and $j_1\ne 0$, giving
\begin{align*}
  \delta(A^{[n]}b) &= \delta(A)\sum_{j_2,\dots, j_\ell\geq 0}
  \prod_{r=2}^\ell \delta_{j_r}(A)+\sum_{j_1=1}^m\sum_{j_2\cdots j_\ell=0}
  \prod_{r=1}^\ell \delta_{j_r}(A).
\end{align*}
The first summand simplifies as follows
\[
\delta(A)\sum_{j_2,\dots, j_\ell\geq 0}
\prod_{r=2}^\ell \delta_{j_r}(A)=\delta(A)\prod_{r=2}^\ell
\sum_{j_r=0}^m \delta_{j_r}(A)=\delta(A)\prod_{r=2}^\ell 1=\delta(A)
\]
and the second summand simplifies as
\[
\sum_{j_1=1}^m\sum_{j_2\cdots j_\ell=0} \prod_{r=1}^\ell
\delta_{j_r}(A)=\sum_{j_1=1}^n \delta_{j_1}(A) \sum_{j_2\cdots j_\ell=0} \prod_{r=2}^\ell
\delta_{j_r}(A)=(1-\delta(A))\sum_{j_2\cdots j_\ell=0} \prod_{r=2}^\ell
\delta_{j_r}(A)
\]
with $\sum_{j_2\cdots j_\ell=0}\prod_{r=2}^\ell
\delta_{j_r}(A)=1-(1-\delta(A))^{\ell-1}$ by the induction
hypothesis. Together,
\[\delta(A^{[n]}b)=\delta(A) + (1-\delta(A))(1-(1-\delta(A))^{\ell-1})=1-(1-\delta(A))^{\ell},\]
as claimed.

\noindent (b) Observe that both formulas in (a) depend only on $k$,
and the number $\ell$  of cycles in the cycle decomposition of $b\in B$.
Therefore
\begin{align*}
    |{\rm Fix}_k(A\power B)|&=\sum_{\ell=1}^n \stir{B}{\ell} |A|^n \sum_{j_1\cdots j_\ell=k} \prod_{r=1}^\ell \delta_{j_r}(A),
\end{align*}
and dividing by $|A\power B|=|A|^n|B|$ yields the first formula of~(b).
The formula
for $k=0$ also follows from (a) together with the fact that $\sum_{\ell=1}^n\stir{B}{\ell}=|B|$.
\end{proof}

\vspace*{2ex}

\begin{corollary}\label{C:Pw}
  If $A\le\sym_m$ and $B\le\sym_n$, then
  $\delta_1(A\power B)=\mathcal{C}_B(\delta_1(A))$ where
  $\mathcal{C}_B(x)=\tfrac{1}{|B|}\sum_{\ell=1}^n\stir{B}{\ell}x^{\ell}$.
\end{corollary}

\subsection{Derangements of \texorpdfstring{$A\power\sym_n$}{}}
If $c_n(g)$ denotes the number of cycles of $g\in\sym_n$, then the formula of Theorem \ref{boston}(b) can be written~as 
\begin{equation*}\label{eq_power2}
    \delta(A\power B)
    = 1-\frac{1}{|B|}\sum_{b\in B} (1-\delta(A))^{c_n(b)}.
\end{equation*}

For $B=C_n=\langle(1,2,\dots,n)\rangle\leq \sym_n$, we obtain the following 
by Example~\ref{example:stir} and Theorem~\ref{boston}(b).
   
\begin{corollary}\label{cor_cyclic}
  Let $A\leq \sym_m$ and let $C_n\leq \sym_n$ be generated by an $n$-cycle. Then
  \[
  \delta(A\power C_n)=1-\frac{1}{n}\sum_{d\mid n} \phi(d) (1-\delta(A))^{n/d},\]
where $\phi$ denotes Euler's $\phi$-function.
\end{corollary}

\begin{corollary}\label{C:AwrSn}
  If $m,n\ge2$ and $A\le \sym_m$, then
  \[
  \delta(A\power \sym_n)=1-\prod_{\ell=1}^n\left(1-\frac{\delta(A)}{\ell}\right)
  \quad\text{and hence $\lim_{n\to\infty}\delta(A\power \sym_n)=1$.}
  \]
  We have
  $1-\frac{1}{n^{\delta(A)}}\le\delta(A\power \sym_n)$, and if
  $A\le\sym_m$ is 2-transitive, then $\delta(A\power \sym_n)\leq 1-\frac{1}{n+1}$.
\end{corollary}
\begin{proof}
  Let $A\power \sym_n$ act on $[m]^{[n]}$ via \Pw. The identity
  $\sum_{\ell=1}^n\stir{n}{\ell}x^\ell=\prod_{\ell=1}^n(x+\ell-1)$,
  see \cite[Table~264]{GKP89}, and Theorem~\ref{boston}(b) yield
  \begin{align*}
    \delta(A\power \sym_n)
    &= 1-\frac{1}{|\sym_n|}\sum_{\ell=1}^n \stir{\sym_n}{\ell}(1-\delta(A))^\ell
    = 1-\frac{1}{n!}\sum_{\ell=1}^n \stir{n}{\ell}(1-\delta(A))^\ell\\
    &=1-\prod_{\ell=1}^{n}\frac{(1-\delta(A))+\ell-1}{\ell}
    =1-\prod_{\ell=1}^n \left(1-\frac{\delta(A)}{\ell}\right).
  \end{align*}
  This proves the first claim. Moreover,
  $\lim_{n\to\infty}\delta(A\power \sym_n)=1$ follows from
  $\prod_{\ell=1}^\infty \left(1-\frac{\delta(A)}{\ell}\right)=0$, which  is
  a consequence of $\sum_{\ell=1}^\infty\delta(A)/\ell=\infty$, see~\cite[Theorem~2.2.2]{Little}.

  We now bound $\delta(A\power\sym_n)$.
  It follows from \cite[Eq.\ (2.2.2)]{Little} that 
  $1-\prod_{\ell=1}^n(1-\frac{\delta(A)}{\ell})\ge 1-e^{-\delta(A)H_n}$ where $H_n=\sum_{\ell=1}^n\frac{1}{\ell}$.  Now $H_n>\log(n)$ yields $1-\prod_{\ell=1}^n(1-\frac{\delta(A)}{\ell})\ge 1-e^{-\delta(A)\log(n)}
  = 1-n^{-\delta(A)}$, as desired. For the last claim, let  $A\le\sym_m$ be 2-transitive. By \eqref{eq_DFG}, it suffices to show that  $G=A\power \sym_n$ has rank $r=n+1$.
 {We now prove this fact.} Let $C$ be the stabiliser in $A$ of
  the point $m$. Then $|A:C|=m$ and hence the stabiliser of the point
  $[m,m,\dots,m]\in [m]^{[n]}$ is $D=C\power \sym_n$.
  Since $|G:D|=m^n$ we see that $G$ is transitive on $[m]^{[n]}$.
  Consider the orbits of $D$ on  $[m]^{[n]}$. If 
  $\omega=(\omega_1,\dots,\omega_n)\in [m]^{[n]}$, then by 2-transitivity we may choose
  $\gamma\in C^{[n]}$ such that  $\omega_i\gamma_i=m$ if $\omega_i=m$, and $\omega_i\gamma_i=1$ otherwise, for every $i\in[n]$. If $\omega_i=m$ for exactly $j$ indices $i\in[n]$, then we may choose $(\gamma,b)\in C\power \sym_n$ such that
  $\omega(\gamma,b)=(1,\dots,1,m,\dots,m)$ with exactly
  $j$ copies of $m$ and $n-j$ copies of $1$. Thus, there are $n+1$ orbits in $[m]^{[n]}$ under the action of $D$, one for each $j\in\{0,1,\dots,n\}$, and therefore $G$ has rank $n+1$, as claimed. 
\end{proof}

\begin{remark}
The proof of Corollary~\ref{C:AwrSn} can be modified to prove that
$\lim_{n\to\infty}\delta(A\power \alt_n)=1$ for all $A\le\sym_m$
with $m\ge2$ by using $\stir{n}{\ell}=\stir{\sym_n}{\ell}\ge\stir{\alt_n}{\ell}$ in the
above display.
\end{remark}

\subsection{Density in \texorpdfstring{$[\delta(A),1]$}{}:
  proof of Theorem~\ref{thm_power_dense1}}
  
We now prove Theorem~\ref{thm_power_dense1}.

\begin{proof}[Proof of Theorem \ref{thm_power_dense1}]
  Suppose first that $B\leq \sym_n$ is fixed. Theorem~\ref{boston}(b) yields
  $\delta(A\power B)=f(\delta(A))$ where 
  $f(x)=1-\mathcal{C}_B(1-x)$. Since $\mathcal{C}_B(x)$ is continuous and
  increasing on $[0,1]$, the same is true for $f(x)$. The set
  $\{\delta(A)\mid \textup{$A$ is primitive}\}$ is dense in $[0,1]$ by
  Theorem~\ref{boston}(c), so the same is true for the set
  $\{f(\delta(A))\mid \textup{$A$ is primitive}\}$.
  This proves the first claim as $f(0)=0$ and $f(1)=1$.

  Suppose next that $A\leq \sym_m$ is fixed.
  For each $B\leq \sym_n$ we have 
    $\mathcal{C}_B(x)\le\frac{1}{|B|}\sum_{\ell=1}^n \stir{B}{\ell}x=x$
    for $x\in[0,1]$.
    It follows from $0\le \mathcal{C}_B(x)\le x\le 1$ that
    $0\le x\le 1-\mathcal{C}_B(1-x)\le1$ and hence by Theorem~\ref{boston}(b)
    that $\delta(A\power B)=1-\mathcal{C}_B(1-\delta(A))\ge\delta(A)$.

    Let $p_1<p_2<\cdots$ be a sequence of primes. Set $n=p_1\cdots p_r$
  and let $Z_r$ be the cyclic and regular subgroup $Z_r \leq \sym_n$ generated by an $n$-cycle.   Let $z=1-\delta(A)$. When $r=1$, and $n=p_1$ we have
  \[
    \frac{1}{n}\sum_{d\mid n}\phi(d)z^{n/d}=\left(1-\frac{1}{p_1}\right)z+\frac{1}{p_1}z^{p_1},
  \]
  by Example~\ref{example:stir}. Since $0<z\le1$,
  Corollary~\ref{cor_cyclic} shows that 
  \[
  \lim_{p_1\to \infty}\delta(A\power Z_1)
  = \lim_{p_1\to \infty} \left(1-\left(1-\frac{1}{p_1}\right)z-\frac{1}{p_1}z^{p_1}\right)
  =1-z=\delta(A).
  \]

  Let $q=p_{r+1}$ be a prime with $p_r<q$. We consider $Z_r\le\sym_n$ and
  $Z_{r+1}\le\sym_{nq}$ where $n=p_1\cdots p_r$ and calculate the difference
  $|\delta(A\power Z_{r+1})-\delta(A\power  Z_r)|$.
  Since $\phi(nq)=(q-1)\phi(n)$,
  Corollary~\ref{cor_cyclic} shows the
  following 
  \begin{align*}
    \delta(A\power Z_{r+1})
    &= 1-\frac{1}{nq}\left(\sum_{d\mid n} \phi(d) z^{nq/d}+\sum_{d\mid n} \phi(dq) z^{nq/dq}\right)\\
    &= 1-\frac{1}{nq}\left(\sum_{d\mid n} \phi(d) z^{n/d}\left((z^{n/d})^{q-1}+(q-1)\right)\right) = \delta(A\power Z_r) + D(n,q),
  \end{align*}
  where
  \[ 
  D(n,q)
  =\frac{1}{nq}\left(\sum_{d\mid n} \phi(d) z^{n/d}
  \left(1-(z^{n/d})^{q-1}\right)\right).
  \]
  Since $0<1-(z^{n/d})^{q-1}<1$ we have
  $0<D(n,q)< \frac{1}{q}\delta(A\power Z_r)$, and so $D(n,q)\to 0$ as
  $q\to \infty$. In conclusion,  the sequence
  $(\delta(A\power Z_r))_{r\ge1}$ is strictly increasing and can be arranged to
  start arbitrarily close to $\delta(A)$ with arbitrarily small step
  size $D(n,q)$.

  It remains to show that this sequence converges to $1$. We focus
  on $Z_r$ and for $\Delta\subseteq [r]$ define $p(\Delta)=\prod_{i\in\Delta}p_i$
  and $\Delta'=[r]\setminus \Delta$. Since $p([r])=p(\Delta)p(\Delta')=n$, we have
  \begin{align*}
    \delta(A\power Z_r)&= 1-\frac{1}{n} \sum_{d\mid n}\phi(d)z^{n/d} =1-\frac{1}{p([r])}\sum_{\Delta\subseteq [r]} \phi(p({\Delta})) z^{p({\Delta'})}\\&=1-\frac{\phi(p([r]))}{p([r])} z - \sum_{\Delta\subseteq [r]\atop \Delta\ne [r]} \frac{\phi(p({\Delta}))}{p({\Delta})} \frac{z^{p({\Delta'})}}{p({\Delta'})}\\
    &\geq 1-\prod_{i=1}^r\left(1-\frac{1}{p_i}\right)z - \sum_{\Delta\subseteq [r]\atop \Delta\ne [r]} \frac{z^{p({\Delta'})}}{p({\Delta'})}
    > 1-\prod_{i=1}^r\left(1-\frac{1}{p_i}\right)z - \sum_{i\geq p_1} \frac{z^i}{i}.
  \end{align*}
  The Taylor series $\sum_{i=1}^\infty z^i/i $ converges to
  $-\log(1-z)$ since $|z|<1$. Thus, for any
  given $\varepsilon>0$ we can choose $p_1$ large enough so that
  $\sum_{i\geq p_1} z^i/i<\varepsilon/2$. Moreover,
  $\prod_{p\text{ prime}} (1-\frac{1}{p})$ diverges to~$0$
  by~\cite[Theorem 2.2.2]{Little} since $\sum_{p\text{ prime}}1/p$ diverges.
  Thus, we can choose $r$ large enough so that
  $\prod_{i=1}^r\left(1-\frac{1}{p_i}\right)<\varepsilon/2z$. The claim now
  follows from
  \[
  \delta(A\power Z_r)
  >1-\frac{\varepsilon}{2z}z-\frac{\varepsilon}{2}=1-\varepsilon.
  \qedhere\]
 \end{proof}

\subsection{Proof of Theorem \ref{thm_power_new1}}

{

\begin{proof}[Proof of Theorem \ref{thm_power_new1}]
  Let $T$ be a transversal for $C_m$ in $A_m$, so that $\mathcal{T}=T^{[n]}$
  is a transversal for $C_m^{[n]}$ in $A_m^{[n]}$. Every element in
  $G_m$ can be written as $(\beta\tau,b)$ where $\beta\in C_m^{[n]}$,
  $\tau\in \mathcal{T}$, and $b\in B$. We fix $(\tau,b)$ and determine the
  proportion $\delta(C_m^{[n]}(\tau,b))$ in the coset $C_m^{[n]}(\tau,b)$. Write $\tau=(\tau_1,\dots,\tau_n)$
  and let $b=b_1\cdots b_\ell$ be the disjoint cycle decomposition of $b$.  Let
  \[
    D_\tau(b)=D_\tau(b_1)\cup\cdots\cup D_\tau(b_\ell)
    \quad\textup{where}\quad
    D_\tau(b_i)=\{ \alpha\in C_m^{[n]} \mid
    b_i(\alpha\tau)\in \Der(\sym_m)\}
  \]
  with $b_i(\alpha\tau)$ as in Definition \ref{def:biproduct}.
  Arguing as in the proof of Theorem~\ref{thm_power}(a), it follows that the 
  number of $\alpha\in C_m^{[n]}$ for which $(\alpha\tau,b)$ is a derangement in $C_m^{[n]}(\tau,b)$
  is $|D_\tau(b)|$. In the next paragraph we  now show that if $b$ has exactly $\ell$ cycles, then 
  \[\lim_{m\to\infty}\delta(C_m^{[n]}(\tau,b))=1-(1-\delta_0)^\ell,\]independent of $\tau$.
 
  Before computing this number via inclusion-exclusion,
  we focus on one of the cycles $b_i$ and count $|D_\tau(b_i)|$. To simplify
  notation, we conjugate by an element of $\sym_m\power\sym_n$ (which
  preserves derangements and non-derangements),
  so we can assume that $b_i=(1,2,\dots,k_i)$. The $b_i$-product
  of $\alpha\tau$ with $\alpha\in C_m^{[n]}$ now is
\[
b_i(\alpha\tau)=\alpha_1\tau_1\alpha_2\tau_2\cdots\alpha_{k_i}\tau_{k_i}=\alpha_1
\left(\alpha_2^{\tau_1^{-1}}\alpha_3^{(\tau_1\tau_2)^{-1}}\cdots
\alpha_{k_i}^{(\tau_1\cdots\tau_{k_i})^{-1}}\right)\tau_1\tau_2\cdots
\tau_{k_i},\] which is a derangement in $C_m\tau_1\cdots\tau_{k_i}$ if
and only if $\alpha_1=x
\left(\alpha_2^{\tau_1^{-1}}\alpha_3^{(\tau_1\tau_2)^{-1}}\cdots
\alpha_{k_i}^{(\tau_1\cdots\tau_{k_i})^{-1}}\right)^{-1}$ with $x\in
C_m$ such that $x\tau_1\cdots\tau_{k_i}\in\Der(C_m\tau_1\cdots
\tau_{k_i})$ and $\alpha_j\in C_m$ for each $j\in [n]\setminus \{1\}$.
Since $b_i(\tau)=\tau_1\cdots \tau_{k_i}$, we have
$|\Der_\tau(b_i)|=|\Der(C_m^{[n]}(\tau,b_i))|=|\Der(C_mb_i(\tau))||C_m|^{n-1}$.
Hence
\[
\frac{|\Der_\tau(b_i)|}{|C_m|^{n}}=\frac{|\Der(C_mb_i(\tau))|}{|C_m|}\to\delta_0
\quad\textup{as $m\to\infty$.}
\]
Since the cycles $b_1,\ldots,b_\ell$ have disjoint support, the previous argument shows that for distinct  elements $i_1,\dots,i_j\in [\ell]$, we have
\[
|D_\tau(b_{i_1})\cap\dots\cap D_\tau(b_{i_j})|
={\left(\prod_{k=1}^j|\Der(C_mb_{i_k}(\tau))|\right)|C_m|^{n-j}}
={\left(\prod_{k=1}^j\frac{|\Der(C_mb_{i_k}(\tau))|}{|C_m|}\right)|C_m|^{n},}
\]
and hence
\[
\frac{|D_\tau(b_{i_1})\cap\dots\cap D_\tau(b_{i_j})|}{|C_m|^{n}}
=\left(\prod_{k=1}^j\frac{|\Der(C_mb_{i_k}(\tau))|}{|C_m|}\right)\to\delta_0^j
\quad\textup{as $m\to\infty$.}
\]
Inclusion-exclusion now shows that
 \begin{align}\label{eq_old}
   \frac{|\Der(C_m^{[n]}(\tau,b))|}{|C_m|^{n}}
   &=\frac{1}{|C_m|^{n}}\left|\bigcup_{i=1}^\ell  D_\tau(b_i)\right|
   =\sum_{j=1}^\ell(-1)^{j-1}\sum_{1\le  i_1<\dots<i_j\le\ell}
     \frac{|D_\tau(b_{i_1})\cap\dots\cap D_\tau(b_{i_j})|}{|C_m|^n}.
 \end{align}
 Taking the limit as $m\to\infty$ gives
 \begin{align}\label{eq_bt}
   \lim_{m\to\infty}\frac{|\Der(C_m^{[n]}(\tau,b))|}{|C_m|^{n}}
   &=\sum_{j=1}^\ell(-1)^{j-1}\sum_{1\le  i_1<\dots<i_j\le\ell}\delta_0^j
     =\sum_{j=1}^\ell(-1)^{j-1}\binom{\ell}{j}\delta_0^j
     =1-(1-\delta_0)^\ell.
 \end{align}
 In summary, if $b$ has exactly $\ell$ cycles, then 
 $\lim_{m\to\infty}\delta(C_m^{[n]}(\tau,b))=1-(1-\delta_0)^\ell$, independent of $\tau$ and the precise structure of the $\ell$ cycles of $b$.

  By assumption, $\pi_n(G_m)=B\leq \sym_n$ for each $m$. The
  kernel of the restriction of $\pi_n$ to $G_m$ is the normal subgroup
  $U_m=G_m\cap \sym_m^{[n]}=G_m\cap A_m^{[n]}$. In particular, $G_m/U_m\cong B$ by the Isomorphism Theorem. By assumption, $C_m^{[n]}\lhdeq U_m$. It follows that for every $b\in B$ with exactly $\ell$ cycles, there exist $|U_m|$ elements in $G_m$ of the form $(\beta,b)$ with $\beta\in A_m^{[n]}$. Moreover, there exist $|U_m|/|C_m|^n$ coset representatives of $C_m^{[n]}$ in $G_m$ of this form, that is, there are  $|U_m/C_m^{[n]}|$
  choices for $\tau$. 
  Recall that $B$ has exactly $\stir{B}{\ell}$ elements with precisely $\ell$ cycles, and let $b(\ell)$ be a fixed element with this property {(and let $b(\ell)$ be arbitrary if $\stir{B}{\ell}=0$)}. By Equation \eqref{eq_bt}, the value of the limit $\lim_{m\to\infty} |\Der(C_m^{[n]}(1,b(\ell)))|/|C_m|^{n}$ is independent of the precise structure of $b(\ell)$ and the choice of $\tau=1$. Thus, using the
  equations $|G_m|=|B||U_m|$,  $\frac{1}{|B|} \sum_{\ell=1}^n \stir{B}{\ell}=1$, 
  and $\mathcal{C}_B(x)=\frac{1}{|B|}\sum_{\ell=1}^n\stir{B}{\ell}x^\ell$
  we obtain 
  \begin{align*}
    \lim_{m\to \infty} \delta(G_m) &=\lim_{m\to\infty}\frac{|\Der(G_m)|}{|G_m|}\\
    & = \lim_{m\to\infty}\frac{1}{|G_m|}\sum_{\ell=1}^n \stir{B}{\ell}\frac{|U_m|}{|C_m|^n}|\Der(C_m^{[n]}(1,b(\ell)))|\\
    &= \frac{1}{|G_m|}\sum_{\ell=1}^n \stir{B}{\ell}|U_m| (1-(1-\delta_0)^\ell) \\
    &=1-\mathcal{C}_B(1-\delta_0).\qedhere
  \end{align*} 
\end{proof}

\vspace*{2ex}

\begin{corollary}\label{C:large}
  Fix $n\ge2$ and let $G_1, G_2, \dots$ be a sequence of subgroups
  where each $G_m$ satisfies
  $\alt_m^{[n]}\lhdeq G_m\le\sym_m\power\sym_n$ and $\pi_n(G_m)=B\le\sym_n$
  is independent of~$m$. Then\[
   \lim_{m\to\infty}\delta(G_m)
   =1-\frac{1}{|B|}\sum_{\ell=1}^n\stir{B}{\ell}(1-e^{-1})^\ell> e^{-1}.
  \]
\end{corollary}
  
\begin{proof}
  The first equality follows from Theorem \ref{thm_power_new1} by
  choosing $C_m=\alt_m$ and $A_m=\sym_m$, and the observation that
  $\lim_{m\to\infty}\delta(\alt_m)=\lim_{m\to\infty}\delta(\sym_m\setminus\alt_m)=e^{-1}$
  by Corollary~\ref{OLDAnSn}(b,c), so we set $\delta_0=e^{-1}$ in Theorem \ref{thm_power_new1}. The inequality follows from
  $\sum_{\ell=1}^n\stir{B}{\ell}(1-e^{-1})^\ell< \sum_{\ell=1}^n\stir{B}{\ell}(1-e^{-1})=(1-e^{-1})|B|$.
\end{proof}
}

A modification of the proof of Theorem \ref{thm_power_new1} also shows the following.
\begin{corollary}
 Suppose $A\leq \sym_m$ and $B\leq \sym_n$ and let $C\lhdeq A$.  Let $\delta_L,\delta_U\in[0,1]$ such $\delta_L\le\delta(Ca)\le\delta_U$
 for all $a\in A$.  Suppose $G\leq \sym_m\power \sym_n$ satisfies $C^{[n]}\lhdeq G \leq A\power B$ and $B=\pi_n(G)$ is the image of the natural projection $\pi_n\colon \sym_m\power\sym_n\to\sym_n$. We have
 \[ 1-x-y \leq \delta(G)\leq 1-x+y\]
 where $x=\tfrac{1}{2}\big(\mathcal{C}_B(1-\delta_U)+\mathcal{C}_B(1-\delta_L)\big)$ and  $y=\tfrac{1}{2}\big(\mathcal{C}_B(1+\delta_U) -\mathcal{C}_B(1+\delta_L)\big)$.

\end{corollary}
  
\begin{proof}
  Starting as in the proof of Theorem \ref{thm_power_new1}, this time we estimate 
  \[|C|^{n}\delta_L^j\le |D_\tau(b_{i_1})\cap\dots\cap D_\tau(b_{i_j})|\le |C|^{n}\delta_U^j.\]
  If $b\in B$ has $\ell$ cycles, this yields the following lower bound in  Equation \eqref{eq_old}:
\begin{align*}
   \frac{|\Der(C^{[n]}(\tau,b))|}{|C|^{n}}
   & 
   =\sum_{j=1}^\ell(-1)^{j-1}\sum_{1\le  i_1<\dots<i_j\le\ell}
   \frac{|D_\tau(b_{i_1})\cap\dots\cap D_\tau(b_{i_j})|}{|C|^n}\\
   &\geq  1 - \sum_{j\geq 0\text{ even}} {\ell \choose j} \delta_U^j + \sum_{j\geq 0\text{ odd}} {\ell \choose j} \delta_L^j\\
   &= 1 - \frac{(1+\delta_U)^\ell + (1-\delta_U)^\ell}{2} + \frac{(1+\delta_L)^\ell-(1-\delta_L)^\ell}{2}.
 \end{align*}
Using this in the last equations of the proof of Theorem \ref{thm_power_new1} yields
\[\delta(G)\geq 1- \frac{1}{2}\big(\mathcal{C}_B(1-\delta_U)+\mathcal{C}_B(1-\delta_L)+\mathcal{C}_B(1+\delta_U) -\mathcal{C}_B(1+\delta_L)\big).\] 
The upper bound follows analogously by swapping the subscripts $U$ and $L$.  
\end{proof}

\section*{Acknowledgements}
We would like to thank Peter Cameron and the two anonymous referees for
suggesting some changes and the inclusion of certain references.

\enlargethispage{1cm}

\vspace*{0.5cm}

\end{document}